\documentclass[11pt]{amsart}
\usepackage{mathrsfs}
\usepackage{amssymb}
\usepackage[T1]{fontenc}
\usepackage{amscd}
\usepackage{amsmath, amssymb}
\usepackage{amsthm}
\usepackage{amsfonts}
\usepackage[colorlinks,linkcolor=blue,citecolor=blue, pdfstartview=FitH]
{hyperref}
\usepackage{amscd}
\usepackage{easybmat}
\usepackage{mathrsfs}
\usepackage{amsfonts}
\usepackage{color}
\usepackage{pifont}
\usepackage{upgreek}
\usepackage{bm}
\usepackage{hyperref}
\usepackage{shorttoc}
\usepackage{amsmath,amstext,amsthm,a4,amssymb,amscd}
\usepackage[mathscr]{eucal}
\usepackage{mathrsfs}
\usepackage{epsf}

\numberwithin{equation}{section}

\def\p{\partial}

\def\b{\bar}
\def\mb{\mathbb}
\def\mc{\mathcal}
\def\n{\nabla}

\def\R{\mathbb R}
\def\C{\mathbb C}
\def\ra{\righarrow}
\def\H{\mathbb H}
\def\ra{\rightarrow}
\theoremstyle{plain}
\newtheorem{thm}{Theorem}[section]
\newtheorem{lemma}[thm]{Lemma}
\newtheorem{prop}[thm]{Proposition}
\newtheorem{cor}[thm]{Corollary}
\theoremstyle{definition}
\newtheorem{rem}[thm]{Remark}

\theoremstyle{definition}

\newcommand{\comment}[1]{}

\let \cal \mathcal

\newenvironment{aligns}{\equation\aligned}{\endaligned\endequation}

\begin{document}

\title{New K\"ahler  metric on quasifuchsian space and its curvature properties}

\makeatletter

\makeatother
\author{InKang Kim}
\author{Xueyuan Wan}
\author{Genkai Zhang}

\address{Inkang Kim: School of Mathematics, KIAS, Heogiro 85, Dongdaemun-gu Seoul, 02455, Republic of Korea}
\email{inkang@kias.re.kr}

\address{Xueyuan Wan: Mathematical Sciences, Chalmers University of Technology and Mathematical Sciences, G\"oteborg University, SE-41296 G\"oteborg, Sweden}
\email{xwan@chalmers.se}

\address{Genkai Zhang: Mathematical Sciences, Chalmers University of Technology and Mathematical Sciences, G\"oteborg University, SE-41296 G\"oteborg, Sweden}
\email{genkai@chalmers.se}

\begin{abstract}
  {
    Let $QF(S)$
    be the quasifuchsian space
    of a closed surface $S$ of genus $g\geq 2$. We construct a new
    mapping class group invariant K\"ahler metric on $QF(S)$.
It is an extension of the Weil-Petersson metric on
the Teichm\"uller space $\mathcal T(S)\subset QF(S)$.
We also
calculate its curvature and
prove 
some negativity for the curvature along the tautological
directions}.
 \end{abstract}
\footnotetext[3]{{Research by Inkang Kim is partially supported by Grant NRF-2017R1A2A2A05001002 and  research by Genkai Zhang is
  partially supported by Swedish
  Research Council (VR).}}
\maketitle
\tableofcontents

\section*{Introduction}
Teichm\"uller space $\cal T(S)$ carries a natural mapping class group invariant K\"ahler metric, called a Weil-Petersson metric $g_{WP}$. There have been  active studies on the properties of this metric since its birth.
More recently, some 
new K\"ahler metrics with more desirable properties
such as K\"ahler hyperbolicity have been found where the
K\"ahler hyperbolicity means that the K\"ahler metric is complete with  bounded curvatures
and it has a bounded  K\"ahler primitive.
Such  K\"ahler hyperbolic metrics are studied by
McMullen \cite{Mc} and Liu-Sun-Yau \cite{Yau}.

In Kleinian group theory, the quasifuchsian space $QF(S)$ is a quasiconformal deformation space of the Fuchsian space $F(S)$ which can be identified with $\cal T(S)$. By Bers' simultaneous uniformization theorem, $QF(S)$ can be naturally identified with $\cal T(S)\times \cal T(\bar S)$ where $\bar S$ is a surface with an orientation reversed. With this identification, the mapping class group acts diagonally on $QF(S)$ and $F(S)=\cal T(S)$ sits diagonally on $\cal T(S)\times \cal T(\bar S)$. But this diagonal embedding is totally real. Hence if one gives a product K\"ahler metric on $QF(S)$, this metric is not an extension of a K\"ahler metric on $F(S)$.
There have been several attempts to extend a K\"ahler metric of $\cal T(S)$ to $QF(S)$.
Bridgeman and Taylor \cite{BT} described a quasi-metric which extends the K\"ahler metric of $\cal T(S)$ but
it vanishes along the pure bending deformation vectors \cite{Br}.

In this article, we give a completely new mapping class group invariant K\"ahler metric on $QF(S)$ which extends any K\"ahler metric on $\cal T(S)$. Indeed such  a metric is already defined in the paper \cite{KZ} a few years ago.
The metric is defined by a K\"ahler potential which is a combination of $L^2$ norm of a fiber and a K\"ahler potential on the base $\cal T(S)$. We will see that $QF(S)$ can be embedded, via Bers embedding using the complex projective structures, in the holomorphic bundle over $\cal T(S)$
with fibers being quadratic holomorphic differentials as a bounded open neighborhood of the zero section.
The K\"ahler metric we construct is the restriction to this open neighborhood.  We choose then the Weil-Petersson metric on $\cal T(S)$ and
show that the new K\"ahler metric on $QF(S)$ has similar properties such as its K\"ahler form has a bounded primitive 
and the curvature has non-positivity for some directions.

\begin{thm}\label{main theorem} There exists a mapping class group invariant  K\"ahler metric on $QF(S)$ which extends the Weil-Petersson metric on $\cal T(S)\subset QF(S)$. Furthermore the curvature of the metric is non-positive when evaluated on the tautological section (and vanishes along vertical directions),  its Ricci curvature is bounded from above by $-\frac{1}{\pi(g-1)}$ when restricted to Teichm\"uller space, and its K\"ahler form has a bounded primitive.
\end{thm}


\vspace{5mm}
{\bf Acknowledgment:} The first author thanks C. McMullen for the communications on K\"ahler metrics on Teichm\"uller space and the suggestions.

\section{Preliminaries}
\subsection{Quasifuchsian space}
Recall that the isometry group
of the hyperbolic 3-space $\H^3$ can be identified
with $PSL(2, \C)$. We use the
unit ball
in $\R^3$
as a realization of $\mathbb H^3$.
The ideal boundary is then $S^2$
and is further  identified with $\mathbb{CP}^1$
such that the action of $PSL(2,\C)$  on $S^2$
is the natural
extension of its isometric action on $\H^3$.

The Teichm\"u{}ller space
$\cal T(S)$ is realized as
the space of Fuchsian representations, i.e., 
discrete and faithful representations
$\rho:\pi_1(S)\rightarrow PSL(2,\R)$ up to conjugacy. Let $\Gamma_\rho$ be the image of $\rho$, whence $\Gamma_\rho$ acts on $S^2$ by M\"obius map preserving the equator. Then  any quasiconformal map $f$ from $S^2$ into itself induces a quasiconformal deformation $\rho_f$ defined by
$$\rho_f(\gamma)=     f \circ \rho(\gamma)\circ f^{-1}.$$
If furthermore  $\rho_f(\gamma)$ is an element of $PSL(2,\C)$ for any $\gamma\in \pi_1(S)$ then it defines
a representation of $\pi_1(S)$ in $PSL(2,\C)$.
Collection of such quasiconformal deformations of Fuchsian
representations is denoted $QF(S)$ and is identified with an open set
of a character variety
 $\chi(\pi_1(S), PSL(2,\C))$.
Hence it has a natural induced complex structure from $\chi(\pi_1(S), PSL(2,\C))$.

If $\phi:\pi_1(S)\rightarrow  PSL(2,\C)$ is a quasifuchsian representation, then $M_\phi=\H^3/\phi(\pi_1(S))$ is a quasifuchsian hyperbolic 3-manifold which is homeomorphic to $S\times \R$. Then two ideal boundaries of $M_\phi$ define a pairs of points $(X, Y)\in \cal T(S)\times \cal T(\bar S)$.
This is known as Bers' simultaneous uniformization of $QF(S)$ \cite{Bers}. In this case, we denote $M_\phi$ by $Q(X,Y)$.
In this identification, a Fuchsian representation $\rho:\pi_1(S)\ra PSL(2,\R)$ whose quotient $X=\H^2/\rho(\pi_1(S))$ is a point in $\cal T(S)$ gets identified with $(X,\bar X)$.

The mapping class group $Mod(S)$ acts on the space of representations $\rho:\pi_1(S)\ra PSL(2,\C)$ by pre-composition $\phi \rho=\rho\circ \phi_*$ where $\phi\in Mod(S)$ and $\phi_*$ is the induced homomorphism on $\pi_1(S)$. Then $Mod(S)$ acts on $QF(S)=\cal T(S)\times \cal T(\bar S)$ diagonally
$$
\phi \rho=\phi (X, Y)
=(\phi  X, \phi  Y).$$

\subsection{Space of complex projective structures on surface}
A complex projective structure on $S$ is a maximal atlas $\{(\phi_i, U_i)| \phi_i:U_i\ra S^2\}$ whose transition maps $\phi_i\circ \phi_j^{-1}$ are restrictions of M\"obius maps. Then the developing map $dev:\widetilde S\ra S^2$ gives rise to a holonomy representation $\rho:\pi_1(S)\ra PSL(2,\C)$.
We denote the space of marked complex projective structures on $S$ by $\cal P(S)$. Since M\"obius transformations are holomorphic, a projective structure determines a complex structure on $S$. In this way we obtain a forgetful map $$\pi:\cal P(S)\ra \cal T(S).$$ Obviously a Fuchsian representation $\rho:\pi_1(S)\ra PSL(2,\R)\subset PSL(2,\C)$ preserving the equator of $S^2$ gives rise to an obvious projective structure
by identifying $\H^2$ with the upper and lower hemisphere of $S^2$. This gives an embedding
$$\sigma_0:\cal T(S)\ra \cal P(S).$$

More generally, for $X\in \cal T(S)$ and {
  $Z\in \pi^{-1}(X):=P(X)$}, by conformally identifying $\widetilde
X=\H^2$, we obtain a developing map $dev:\H^2\ra S^2=\mathbb{CP}^1$
for $Z$. Hence the developing map can be regarded as a meromorphic
function $f=dev$
on $\H^2$. Then the Schwarzian derivative
$$S(f)=\left[\left(\frac{f''(z)}{f'(z)}\right)'-\frac{1}{2}\left(\frac{f''(z)}{f'(z)}\right)^2\right] dz^2$$ descends to $X$ as a holomorphic quadratic differential. It is known that for any element in holomorphic quadratic differentials $Q(X)$ on $X$, one can show that there exists a complex projective structure over $X$ by solving Schwarzian linear ODE equation.

In this way, $\cal P(S)$ can be identified with a holomorphic vector
bundle $\cal Q(S)$ over $\cal T(S)$ whose fiber over $X$ is
$Q(X)$. In particular this
identifies {$P(X)$} with $Q(X)$ as affine spaces, \cite{Dumas}, and the choice of a base point $Z_0$ in $P(X)$ gives an isomorphism $Z\ra Z-Z_0$. Hence we will choose $Z_0=\sigma_0(X)$ and $\cal T(S)$ will be identified with zero section on $\cal Q(S)$.

\subsection{Embedding of quasifuchsian space into the space of complex projective structures}
Recall that given $X\in \cal T(S), Y\in \cal T(\bar S)$ 
 the Bers' uniformization determines
the quasifuchsian manifold $Q(X,Y)$. 
 Then $Q(X, Y)$ has domain of discontinuity $\Omega_+ \cup \Omega_-$ with
$\Omega_+/ Q(X,Y)=X, \Omega_-/Q(X,Y)= Y$ where $Q(X,Y)$ is viewed as a quasifuchsian representation into $PSL(2,\C)$.

As a quotient of a domain in $\mathbb{CP}^1$ by a discrete group in $PSL(2,\C)$, $\Omega_-/Q(X,Y)$ is a marked projective surface $\Sigma_Y(X)$. Then for a fixed $Y$, we obtain a quasifuchsian section, called a Bers' embedding
$$\beta_Y:\cal T(S)\ra P(Y)\subset \cal P(\bar S).$$
It is known that this map
$$Q(X,Y) \ra \Omega_-/Q(X,Y)$$ is a homeomorphism onto its image in $\cal P(\bar S)$; see e.g. \cite{Dumas}. Under the identification of $\cal P(\bar S)$ with $\cal Q(\bar S)$ such that $\sigma_0(\cal T(\bar S))$ is a zero section,
$$Q(X,Y) \ra \Omega_-/Q(X,Y) - \sigma_0(Y),$$
 this embedding includes zero section which is the image of $\cal T(S)$.

 The space $Q(Y)$ of quadratic differentials is also
 equipped with $L^\infty$-norm defined by
$$||\phi||_\infty=\sup_{Y} \rho^{-2}|\phi(z)|$$ where $\rho(z)|dz|$ is a hyperbolic metric on $Y$.
Then by Nehari's bound \cite{Mc}  we get 
\begin{thm} The above embedding of $Q(X,Y)$ into $Q(Y)$ is contained in a ball of radius $\frac{3}{2}$ in 
$Q(Y)$ where the norm is the $L^\infty$-norm on quadratic differentials.
\end{thm}
\begin{cor}\label{bounded}The quasifuchsian space $QF(S)$ embeds into a neighborhood of a zero section in $\cal Q(\bar S)$
which is contained in a ball of radius $9\pi(g-1)$ in $L^2$-norm on each fiber $Q(Y)$.
\end{cor}
\begin{proof}The $L^2$-norm of a quadratic differential $\phi(z)dz^2$ is given by
$$\int_Y |\phi(z)|^2 \rho(z)^{-4}  \rho(z)^2 |dz|^2\leq ||\phi||_\infty^2 2\pi(2g-2)\leq 9\pi(g-1).$$
\end{proof}

\subsection{Vector bundle isomorphism between quadratic differentials and Beltrami differentials}
The holomorphic tangent bundle of  Teichm\"uller space $\cal T(S)$
is a holomorphic vector bundle over Teichm\"uller space whose fiber over $X$ is the set of harmonic Beltrami differentials $B(X)$. For a harmonic Beltrami differential $\mu(z)\frac{d\bar z}{dz}$ over $X$ with a hyperbolic metric $g=\rho(z)|dz|$, the $L^2$-norm defines the Weil-Petterson metric
$$\Vert \mu\Vert_{WP}^2=\int_X |\mu(z)|^2 \rho(z)^2 |dz|^2$$
on the tangent space of $\cal T(S)$.
 The set of harmonic Beltrami differentials $B(X)$ and $Q(X)$ are
 vector bundle isomorphic by the natural identification of
 differential forms with tangent vectors via the metric,
$$ \Phi=\phi(z) dz^2 \ra \beta=\beta_\Phi=\frac{\overline{\phi(z)}}{\rho^2(z)}\frac{d\bar z}{dz}.$$
The  $L^2$-norms are by definition preserved,
$$\Vert\beta\Vert^2=||\beta||_{WP}^2=\int_X \frac{|\phi(z)|^2}{\rho^4(z)}\rho^2(z) |dz|^2=||\Phi||^2.$$

By Corollary \ref{bounded}, we get
\begin{cor}\label{bound}
Under this isomorphism between cotangent bundle and holomorphic tangent bundle of $\cal T(S)$, the quasifuchsian space $QF(S)$ embeds into a neighborhood of a zero section
in the holomorphic tangent bundle of $\cal T(S)$ which is contained in a ball of radius $9\pi(g-1)$ in $L^2$-norm  on each fiber $B(X)$.
\end{cor}

\section{Griffiths negativity and K\"ahler metric on the holomorphic vector bundles}

\subsection{Griffiths negativity}
As elaborated above the space  $QF(S)$
can be realized as an open set in the tangent bundle
of $\mathcal T(S)$, and we shall construct metrics
on $QF(S)$ using some general constructions.
For that purpose we recall the notion of Griffiths
positivity. Identifying $\cal P(S)$ with the holomorphic vector bundle
$\cal Q(S)$ whose fiber over $Y\in \cal T(S)$ is $Q(Y)$, one can give
a mapping class group invariant K\"ahler metric on $\cal Q(S)$ as
follows. By a theorem of Berndtsson \cite{Bo}, one can show that $\cal
Q(S)$ is Griffiths positive. See \cite{KZ} for a proof.
Hence its dual bundle $\cal B(S)=\cal Q^*(S)$, which is a tangent
bundle of Teichm\"uller space whose fiber is the set of Beltrami
differentials, is Griffiths negative. We fix in the rest of the paper
this realization of $QF(S)$ as
a subset in $\cal Q^*(S)$.
The $L^2$-norm of a Beltrami differential $w=\mu(v) \frac{ d\bar v}{d v}$ is given by
$$||w||^2=(w,w)=\int_Y |\mu(v)|^2 \rho(v)^2|dv|^2$$ where $v$ is a local holomorphic coordinate on $Y$ and $\rho(v)|dv|$ is a hyperbolic metric on $Y$.  Here $(,)$ denotes the $L^2$ inner product over each fiber and $||\cdot||$ denotes its associated norm.

  The K\"ahler metric depending on a constant $k>0$ and a K\"ahler metric on $\cal T(S)$, is constructed on $\cal B(S)=\cal Q^*(S)$ via K\"ahler potential
$$\Phi(w)=||w||^2 + k \pi^* \psi(w),$$ where $w$ is an element in the fiber, $\psi$ is a K\"ahler potential on $\cal T(S)$ and $\pi:\cal B(S)\ra \cal T(S)$ is a projection.

In  local holomorphic coordinates $(z,x)$ around $w_0$, where $z=(z_1,\cdots,z_{3g-3})$ is  local holomorphic coordinates around $\pi(w_0)=z_0$, and $w=\sum x^\alpha e_\alpha(z)$ with respect to  local holomorphic sections $e_\alpha$, for a holomorphic tangent vector at $w$ $T=u+v$ with a canonical decomposition into $\cal T(S)$ direction $u$ and vertical fiber direction $v$, the norm of $T$ with respect to the K\"ahler metric $g$ defined by the K\"ahler potential $\Phi$ is given by
$$||T||_\Phi^2=\bar\partial_T \partial_T \Phi(w)= - (R(u,\bar u)w,w)+ (\cal D_u w+ v, \cal D_u w+v)+k \bar\partial_u
 \partial_u \psi >0,$$ where $R$ is a curvature of the Chern
 connection $\nabla$ on $\cal B(S)$ and $\nabla=\cal D +\bar\partial$
 is a decomposition into $(1,0)$ and $(0,1)$ part of the connection.
 See \cite{KZ} for details.

Since this construction is general, we treat this construction as general as possible in the following subsection.

\subsection{K\"ahler metrics on Griffiths negative vector bundles}

Let $\pi: E\to M$ be a holomorphic vector bundle of rank $r$ over a complex manifold $M$, $\dim M=n$. Let $\{e_i\}_{i=1}^r$ be a local holomorphic frame of $E$ and $\{z^\alpha\}_{\alpha=1}^n$ be local coordinates of $M$.
 Let $G$ be a Hermitian metric on $E$ with Griffiths negative curvature, that is 
\begin{align*}
R_{i\b{j}\alpha\b{\beta}}v^i\b{v}^j\xi^\alpha\b{\xi}^{\beta}<0	
\end{align*}
for any non-zero vectors $v=v^ie_i\in E$ and $\xi=\xi^\alpha\frac{\p}{\p z^\alpha}\in TM$. Here 
\begin{align*}
	R_{i\b{j}\alpha\b{\beta}}=-\p_\alpha\p_{\b{\beta}}G_{i\b{j}}+G^{k\b{l}}\p_\alpha G_{i\b{l}}\p_{\b{\beta}}G_{k\b{j}}
\end{align*}
denotes the Chern curvature tensor of the Hermitian metric $G$. With respect to the local holomorphic frame $\{e_i\}_{i=1}^r$, the complex manifold $E$ is equipped with the following holomorphic coordinates 
\begin{align*}
(z;v)=(z^1,\cdots, z^n; v^1,\cdots, v^r),	
\end{align*}
representating the point $v=v^ie_i\in E$. 
The Hermitian metric $G$ also gives the norm square function
on $E$. By abuse of notation, we also denote it by $G$, i.e.
the function
\begin{align*}
v\in E\mapsto G(v)=G(v,\b{v})=G_{i\b{j}}v^i\b{v}^j.	
\end{align*}
Then $\p\b{\p}G$ is a $(1,1)$-form on $E$. 
\begin{lemma}\label{lemma1}
Denote $\delta v^i:=dv^i+G_{\alpha\b{l}}G^{\b{l}i}dz^\alpha$. Then
	\begin{align*}
	\p\b{\p}G=-R_{i\b{j}\alpha\b{\beta}}v^i\b{v}^jdz^\alpha\wedge d\b{z}^\beta+G_{i\b{j}}\delta v^i\wedge \delta \b{v}^j.	
	\end{align*}
\end{lemma}
\begin{proof}
This follows by a direct computation,
	\begin{align*}
		&\quad -R_{i\b{j}\alpha\b{\beta}}v^i\b{v}^jdz^\alpha\wedge d\b{z}^\beta+G_{i\b{j}}\delta v^i\wedge \delta \b{v}^j\\
		&=-(-\p_\alpha\p_{\b{\beta}}G_{i\b{j}}+G^{\b{l}k}\p_\alpha G_{i\b{l}}\p_{\b{\beta}}G_{k\b{j}})v^i\b{v}^jdz^\alpha\wedge d\b{z}^\beta\\
		&\quad+ G_{i\b{j}}(dv^i+G_{\alpha\b{l}}G^{\b{l}i}dz^\alpha)\wedge (d\b{v}^j+G_{\b{\beta}k}G^{\b{j} k}d\b{z}^\beta)\\
		&=G_{\alpha\b{\beta}}dz^\alpha\wedge d\b{z}^\beta+G_{\alpha\b{j}}dz^\alpha\wedge d\b{v}^j+G_{i\b{\beta}}dv^i\wedge d\b{z}^j+G_{i\b{j}}dv^i\wedge d\b{v}^j\\
		&=\p\b{\p}G.
	\end{align*}
\end{proof}
Now we assume $(M,\omega=\sqrt{-1}g_{\alpha\b{\beta}}dz^\alpha\wedge
d\b{z}^\beta)$ is a K\"ahler manifold and
 $(E,G)$ is Griffiths negative. Define the $(1, 1)$-form
\begin{align}\label{metric}
	\Omega:=\pi^*\omega+\sqrt{-1}\p\b{\p}G.
\end{align}
Lemma \ref{lemma1} then implies that $\Omega$
is a K\"ahler metric on $E$.
In terms of local coordinates $\Omega$ is
\begin{align}\label{metric1}
\Omega=\sqrt{-1}\Omega_{\alpha\b{\beta}}dz^\alpha\wedge d\b{z}^\beta+\sqrt{-1}G_{i\b{j}}\delta v^i\wedge \delta \b{v}^j,
\end{align}
where 
\begin{align}\label{1.2}
	\Omega_{\alpha\b{\beta}}:=-R_{i\b{j}\alpha\b{\beta}}v^i\b{v}^j+g_{\alpha\b{\beta}}
\end{align}
is a positive definite matrix.
The differential  $\p G$ of $G$
is a globally defined one-form on $E$, and its norm square is $G$.
Indeed,
\begin{align*}
\p G=G_{\alpha}dz^\alpha+G_i dv^i=G_i(dv^i+G_{\alpha\b{l}}G^{\b{l}i}dz^\alpha)=G_i\delta v^i.	
\end{align*}
Its norm squre with respect to the metric $\Omega$ is 
\begin{align*}
\|\p G\|^2=G_iG_{\b{j}}G^{\b{j} i}.
\end{align*}
Since $G=G_{i\b{j}}v^i\b{v}^j$, so $G_i=G_{i\b{j}}\b{v}^j$ and
\begin{align*}
	G_iG_{\b{j}}G^{\b{j} i}=G_{i\b{l}}\b{v}^lG_{k\b{j}}v^kG^{\b{j} i}=G_{k\b{l}}v^k\b{v}^l=G,
\end{align*}
which yields that 
\begin{align*}
\|\p G\|^2=G_iG_{\b{j}}G^{\b{j} i}=G,
\end{align*}
which is independent of the metric $\omega$.
\begin{prop}
The norm of the one-form $\p G$ with respect to $\Omega$ is given by 
\begin{align*}
\|\p G\|^2=G	
\end{align*}
	for any metric $\omega$ on $M$. In particular, 
	\begin{align*}
	\|\p G\|^2<R	
	\end{align*}
on the disk bundle $S_R=\{(z,v)\in E| G(z,v)<R\}$.
\end{prop}
As a corollary, we obtain
\begin{cor}\label{cor1}
If $\omega$ is $d$-bounded, $\omega=d\beta$ for some (locally defined)
bounded one-form
$\beta$, then $\Omega$ is also $d$-bounded
on any bounded domain of $E$ with
$\Omega=d(\p G+\pi^*\beta)$ and bounded one-form
$\p G+\pi^*\beta$.
\end{cor}
Specifying to the space
$QF(S)$ we find that
new K\"ahler metric on $QF(S)$ has a bounded primitive if 
the K\"ahler form on $\cal T(S)$ has a bounded primitive.

\section{Curvature of the new K\"ahler metric}

In this section, we will calculate the curvature of the K\"ahler metric $\Omega$ (\ref{metric}). By  \cite[Section 2]{CW}, the Hermitian metric $G$ gives a decomposition on the tangent bundle $TE$ of $E$, i.e. $$TE=\mc{H}\oplus \mc{V},$$
where the horizontal subbundle $\mc{H}$ and vertical subbundle $\mc{V}$ are given by 
$$\mc{H}=\text{Span}_{\mb{C}}\left\{\frac{\delta}{\delta z^\alpha}:=\frac{\p}{\p z^\alpha}-G_{\alpha\b{l}}G^{\b{l}i}\frac{\p}{\p v^i}, 1\leq \alpha\leq n\right\},\quad \mc{V}=\text{Span}_{\mb{C}}\left\{\frac{\p}{\p v^i},1\leq i\leq r\right\}.$$
By duality, the cotangent bundle $T^*E=\mc{H}^*\oplus \mc{V}^*$ with 
\begin{align*}
\mc{H}^*=\text{Span}_{\mb{C}}\left\{dz^\alpha,1\leq \alpha\leq n\right\},\quad \mc{V}^*=\text{Span}_{\mb{C}}\left\{\delta v^i=dv^i+G_{\alpha\b{l}}G^{\b{l}i}dz^\alpha, 1\leq i\leq r\right\}	. 
\end{align*}
 Let $\n=\n'+\b{\p}$ denote the Chern connection of $\Omega$ and  
\begin{align*}
R^{\Omega}=\n^2=\n'\circ \b{\p}+\b{\p}\circ\n'\in A^{1,1}(E, \text{End}(TE))	
\end{align*}
denote the Chern curvature of $\n$. Then 
\begin{aligns}\label{1.3}
\n'\left(\frac{\delta}{\delta z^\alpha}\right) &=	\left\langle\n'\left(\frac{\delta}{\delta z^\alpha}\right),\frac{\delta}{\delta z^\beta}\right\rangle \Omega^{\b{\beta}\gamma}\frac{\delta}{\delta z^\gamma}+\left\langle\n'\left(\frac{\delta}{\delta z^\alpha}\right),\frac{\p}{\p v^j}\right\rangle G^{\b{j}i}\frac{\p}{\p v^i}\\
&=\left(\p \Omega_{\alpha\b{\beta}}-\left\langle\frac{\delta}{\delta z^\alpha},\b{\p}\left(\frac{\delta}{\delta z^\beta}\right)\right\rangle\right)\Omega^{\b{\beta}\gamma}\frac{\delta}{\delta z^\gamma}\\
&=\p \Omega_{\alpha\b{\beta}}\Omega^{\b{\beta}\gamma}\frac{\delta}{\delta z^\gamma},
\end{aligns}
where the last equality holds since $\b{\p}\left(\frac{\delta}{\delta z^\beta}\right)$ is vertical, and 
\begin{aligns}\label{1.4}
\n'\left(\frac{\p}{\p v^i}\right)&=\left\langle\n'\left(\frac{\p}{\p v^i}\right),\frac{\delta}{\delta z^\beta}\right\rangle \Omega^{\b{\beta}\gamma}\frac{\delta}{\delta z^\gamma}+\left\langle\n'\left(\frac{\p}{\p v^i}\right),\frac{\p}{\p v^j}\right\rangle G^{\b{j}k}\frac{\p}{\p v^k}\\
&=-\left\langle\frac{\p}{\p v^i},\b{\p}\left(\frac{\delta}{\delta z^\beta}\right)\right\rangle \Omega^{\b{\beta}\gamma}\frac{\delta}{\delta z^\gamma}+\p G_{i\b{j}} G^{\b{j}k}\frac{\p}{\p v^k}\\
&=G_{i\b{j}}\p(G_{k\b{\beta}}G^{\b{j} k})\Omega^{\b{\beta}\gamma}\frac{\delta}{\delta z^\gamma}+\p G_{i\b{j}} G^{\b{j}k}\frac{\p}{\p v^k}\\
&=G_{i\b{j}}\p_\alpha(G_{k\b{\beta}}G^{\b{j} k})\Omega^{\b{\beta}\gamma}dz^\alpha\otimes\frac{\delta}{\delta z^\gamma}+\p_\alpha G_{i\b{j}} G^{\b{j}k}dz^\alpha\otimes\frac{\p}{\p v^k},
\end{aligns}
where the last equality follows from the fact $G_{i\b{j}k}=0$ since $G_{i\bar j}$ is
a metric along vertical directions.
From (\ref{1.3}) and (\ref{1.4}), the curvature $R^{\Omega}$ is 
\begin{aligns}
R^\Omega\left(\frac{\delta}{\delta z^\alpha}\right)	&=(\n'\circ \b{\p}+\b{\p}\circ\n')\left(\frac{\delta}{\delta z^\alpha}\right)\\
&=\n'\left(-\b{\p}(G_{\alpha\b{l}}G^{\b{l}i})\frac{\p}{\p v^i}\right)+\b{\p}\left(\p \Omega_{\alpha\b{\beta}}\Omega^{\b{\beta}\gamma}\frac{\delta}{\delta z^\gamma}\right)\\
&=\left(-\p\b{\p}(G_{\alpha\b{l}}G^{\b{l}k})-\p G_{i\b{j}}G^{\b{j}k}\wedge \b{\p}(G_{\alpha\b{l}}G^{\b{l}i})+\p\Omega_{\alpha\b{\beta}}\Omega^{\b{\beta}\gamma}\wedge\b{\p}(G_{\gamma\b{l}}G^{\b{l}k})\right)\frac{\p}{\p v^k}\\
&\quad+\left(\b{\p}(\p\Omega_{\alpha\b{\beta}}\Omega^{\b{\beta}\gamma})-\p(G_{k\b{\beta}}G^{\b{j} k})G_{i\b{j}}\Omega^{\b{\beta}\gamma}\wedge \b{\p}(G_{\alpha\b{l}}G^{\b{l}i})\right)\frac{\delta}{\delta z^\gamma},
\end{aligns}
and 
\begin{aligns}\label{1.5}
R^\Omega\left(\frac{\p}{\p v^i}\right)&=\b{\p}\circ\n'	\left(\frac{\p}{\p v^i}\right)\\
&=\b{\p}
\left(G_{i\b{j}}\p(G_{k\b{\beta}}
  G^{\b{j} k})\Omega^{\b{\beta}\gamma}\frac{\delta}{\delta z^\gamma}+\p G_{i\b{j}} G^{\b{j}k}\frac{\p}{\p v^k}\right)\\
&=\b{\p}(G_{i\b{j}}\p(G_{ k\b{\beta}}G^{\b{j}k})\Omega^{\b{\beta}\gamma})\frac{\delta}{\delta z^\gamma}\\
&\quad+\left(G_{i\b{j}}\p(G_{k\b{\beta}}G^{\b{j}k})\Omega^{\b{\beta}\gamma}\wedge \b{\p}(G_{\gamma\b{l}}G^{\b{l}k})+\b{\p}(\p G_{i\b{j}} G^{\b{j}k})\right)\frac{\p}{\p v^k}.
\end{aligns}
Therefore, we obtain
\begin{prop}\label{prop1}
	The Chern curvature $R^{\Omega}$ satisfies 
	\begin{itemize}
	\item[(i)] $\langle R^\Omega\left(\frac{\p}{\p v^i}\right),\frac{\p}{\p v^j}\rangle=(R_{i\b{l}\alpha\b{\beta}}R_{k\b{j}\gamma\b{\sigma}}v^k\b{v}^l\Omega^{\b{\beta}\gamma}+R_{i\b{j}\alpha\b{\sigma}})dz^\alpha\wedge d\b{z}^\sigma$;
	\item[(ii)]	$\langle R^{\Omega}(\frac{\delta}{\delta z^\alpha}),\frac{\delta}{\delta z^\beta}\rangle=\b{\p}(\p\Omega_{\alpha\b{\sigma}}\Omega^{\b{\sigma}\gamma})\Omega_{\gamma\b{\beta}}-R_{p\b{l}\gamma\b{\beta}}R_{k\b{q}\alpha\b{\sigma}}v^k\b{v}^lG^{\b{q} p}dz^\gamma\wedge d\b{z}^\sigma$.
	\end{itemize}
\end{prop}
\begin{proof}
(i) We compute the inner product according to (\ref{1.5}),
\begin{align*}
	\left\langle R^\Omega\left(\frac{\p}{\p v^i}\right),\frac{\p}{\p v^j}\right\rangle &=G_{i\b{q}}\p(G_{k\b{\beta}}G^{\b{q}k})\Omega^{\b{\beta}\gamma}\wedge \b{\p}(G_{\gamma\b{l}}G^{\b{l}k})G_{k\b{j}}+\b{\p}(\p G_{i\b{q}} G^{\b{q}k})G_{k\b{j}}.
\end{align*}
Note that $G_{i\bar j} G^{k\bar j}=\delta_i^k$, hence
$$\p_\alpha(G^{k\bar q}G_{i\bar q})=0=\p_\alpha(G^{k\bar q})G_{i\bar q}+ G^{k\bar q}G_{\alpha i \bar q}.$$
Then
\begin{align*}
G_{i\b{q}}\p(G_{\b{\beta}k}G^{k\b{q}})=&\left[G_{i\b{q}}(\p_\alpha(G_{\b{\beta}k})G^{k\b{q}})+
G_{i\b{q}}(G_{\b{\beta}k}\p_\alpha(G^{k\b{q}}))\right]dz^\alpha\\
=&(G_{i\bar q}G_{\alpha \bar\beta k}G^{k\bar q}-G_{\bar \beta k}G^{k\bar q}G_{\alpha i \bar q})dz^\alpha\\
=& (G_{\alpha\bar \beta i}-G_{\bar \beta k}G^{k\bar q}G_{\alpha i \bar q} )dz^\alpha
\end{align*}
But $R_{i\bar l \alpha\bar\beta}=-\p_\alpha\p_{\b\beta}G_{i\b l}+ G^{k\b j}\p_\alpha G_{i\b j} \p_{\b \beta}G_{k\b l}$
and $G=G(z,v)=G_{i\bar j}(z)v^i \bar v^j$, hence
$$G_i=G_{i\b j}\b v_j, G_{\alpha\b\beta i}=G_{\alpha\b \beta i \b l}\b v^l,$$ and
\begin{align*}
	R_{i\b l\alpha\b \beta}\b v^l &=-G_{\alpha\b\beta i \b l}\b v^l + G^{k\b j}G_{\alpha i \b j}G_{\b \beta k \b l}\b v^l\\
	&=-G_{\alpha\b\beta i}+ G^{k\b j}G_{\alpha i\b j}G_{\b\beta k}\\
	&=-G_{\alpha\b\beta i}+ G^{k\b q}G_{\alpha i\b q}G_{\b\beta k}.
\end{align*}
Finally we get
\begin{align}\label{1.6}
G_{i\b{q}}\p(G_{k\b{\beta}}G^{\b{q} k})=\left(G_{\alpha\b{\beta}k}-G_{k\b{\beta}}G_{i\b{q}\alpha}G^{\b{q} k}	\right)dz^\alpha=-R_{i\b{l}\alpha\b{\beta}}\b{v}^ldz^\alpha.
\end{align}
Similar calculations give
\begin{align*}
\left\langle R^\Omega\left(\frac{\p}{\p v^i}\right),\frac{\p}{\p v^j}\right\rangle=(R_{i\b{l}\alpha\b{\beta}}R_{k\b{j}\gamma\b{\sigma}}v^k\b{v}^l\Omega^{\b{\beta}\gamma}+R_{i\b{j}\alpha\b{\sigma}})dz^\alpha\wedge d\b{z}^\sigma.	
\end{align*}
(ii) Using (\ref{1.6}) we compute
\begin{aligns}
&\quad \left\langle R^{\Omega}\left(\frac{\delta}{\delta z^\alpha}\right),\frac{\delta}{\delta z^\beta}\right\rangle\\
&=\left\langle \left(\b{\p}(\p\Omega_{\alpha\b{\beta}}\Omega^{\b{\beta}\gamma})-
    \p(G_{k\b{\beta}}G^{\b{j} k})G_{i\b{j}}\Omega^{\b{\beta}\gamma}\wedge \b{\p}(G_{\alpha\b{l}}G^{\b{l}i})\right)\frac{\delta}{\delta z^\gamma},\frac{\delta}{\delta z^\beta}\right\rangle\\
&=\b{\p}(\p\Omega_{\alpha\b{\sigma}}\Omega^{\b{\sigma}\gamma})\Omega_{\gamma\b{\beta}}-R_{p\b{l}\gamma\b{\beta}}R_{k\b{q}\alpha\b{\sigma}}v^k\b{v}^lG^{\b{q} p}dz^\gamma\wedge d\b{z}^\sigma.
\end{aligns}
\end{proof}
\begin{rem}
	From (i), when evaluated on a vertical vector,  $\langle R^\Omega\left(\frac{\p}{\p v^i}\right),\frac{\p}{\p v^j}\rangle$ vanishes, that is $\langle R^\Omega(\frac{\p}{\p v^k}, \bullet)\left(\frac{\p}{\p v^i}\right),\frac{\p}{\p v^j}\rangle=0$. 	
\end{rem}
  There exists a canonical holomorphic section of $\mc{V}$, that is 
  $$P=v^i\frac{\p}{\p v^i}\in \mc{O}_E(\mc{V})$$
   which is called the tautological section (see e.g. \cite[Section
   3]{Aikou}). Denote 
   \begin{align*}
   \Psi_{\alpha\b{\beta}}=	-R_{i\b{j}\alpha\b{\beta}}v^i\b{v}^j.
   \end{align*}
   From Proposition \ref{prop1} the $(1, 1)$-form
$   \langle R^\Omega(P),P\rangle$ is
\begin{align*}
  \langle R^\Omega(P),P\rangle
  &=\left\langle R^\Omega\left(\frac{\p}{\p v^i}\right),\frac{\p}{\p v^j}\right\rangle v^i\b{v}^j\\
	&=(R_{i\b{l}\alpha\b{\beta}}R_{k\b{j}\gamma\b{\sigma}}v^k\b{v}^l\Omega^{\b{\beta}\gamma}+R_{i\b{j}\alpha\b{\sigma}})v^i\b{v}^jdz^\alpha\wedge d\b{z}^\sigma\\
	&=\left(\Psi_{\alpha\b{\beta}}\Psi_{\gamma\b{\sigma}}\Omega^{\b{\beta}\gamma}-\Psi_{\alpha\b{\sigma}}\right)dz^\alpha\wedge d\b{z}^{\sigma}.
\end{align*}
For any point $(z, v)$
outside the zero section, i.e in the set
$$E^o:=\{(z,v)\in E; v\neq 0\},$$
the vector  $P(z,v)\neq 0$. So $(\Psi_{\alpha\b{\beta}})$ is a positive definite matrix on  $E^o$ by Griffiths negativity. Since 
\begin{align*}
\Omega_{\alpha\b{\beta}}-\Psi_{\alpha\b{\beta}}=g_{\alpha\b{\beta}}	
\end{align*}
is positive definite,
  so 
  \begin{aligns}\label{1.9}
  	\langle \sqrt{-1}R^\Omega(P),P\rangle &=\sqrt{-1}\left(\Psi_{\alpha\b{\beta}}\Psi_{\gamma\b{\sigma}}\Omega^{\b{\beta}\gamma}-\Psi_{\alpha\b{\sigma}}\right)dz^\alpha\wedge d\b{z}^{\sigma}\\
  	&\leq  \sqrt{-1}\left(\Psi_{\alpha\b{\beta}}\Psi_{\gamma\b{\sigma}}\Psi^{\b{\beta}\gamma}-\Psi_{\alpha\b{\sigma}}\right)dz^\alpha\wedge d\b{z}^{\sigma}=0.
  \end{aligns}
Thus, we obtain
\begin{prop}\label{prop2}
$\langle \sqrt{-1}R^\Omega(P),P\rangle$
is a non-positive $(1,1)$-form on $E$.
\end{prop}
\begin{rem}
Moreover,  $\langle \sqrt{-1}R^\Omega(P),P\rangle$ is a strictly negative $(1,1)$-form on $E^o$ along the horizontal directions, that is 
$$\langle R^\Omega(\xi,\bar{\xi})(P),P\rangle<0$$
for any nonzero vector $\xi=\xi^\alpha\frac{\delta}{\delta z^\alpha}\in \mc{H}_{(z,v)}$, $(z,v)\in E^o$. In fact, from (\ref{1.9}), 
$\langle R^\Omega(\xi,\bar{\xi})(P),P\rangle=0$ if and only if 
\begin{align}\label{1.10}\left(\Psi^{\b{\beta}\gamma}-\Omega^{\b{\beta}\gamma}\right)(\Psi_{\alpha\b{\beta}}\xi^\alpha)(\Psi_{\gamma\b{\sigma}}\b{\xi}^\sigma)=0.\end{align}
Since $\left(\Psi^{\b{\beta}\gamma}-\Omega^{\b{\beta}\gamma}\right)$ is positive definite on $E^o$, so (\ref{1.10}) is equivalent to 
$$\Psi_{\alpha\b{\beta}}\xi^\alpha=0.$$
On the other hand, $(\Psi_{\alpha\b{\beta}})$ is positive definite on $E^o$ by Griffiths negativity of $(E,G)$, which implies that $\xi=0$.  
\end{rem}

The Ricci curvature of the K\"ahler metric is 
\begin{aligns}
  Ric^\Omega&:=\text{Tr}(R^\Omega)=G^{\b{j} i}\left\langle R^\Omega\left(\frac{\p}{\p v^i}\right),\frac{\p}{\p v^j}\right\rangle+
  \Omega^{\b{\beta}\alpha
  }\left\langle R^{\Omega}\left(\frac{\delta}{\delta z^\alpha}\right),\frac{\delta}{\delta z^\beta}\right\rangle\\
&=G^{\b{j} i}(R_{i\b{l}\alpha\b{\beta}}R_{k\b{j}\gamma\b{\sigma}}v^k\b{v}^l\Omega^{\b{\beta}\gamma}+R_{i\b{j}\alpha\b{\sigma}})dz^\alpha\wedge d\b{z}^\sigma\\
&\quad+\Omega^{
  \b{\beta}
  \alpha
}
(\b{\p}(\p\Omega_{\alpha\b{\sigma}}\Omega^{\b{\sigma}\gamma})\Omega_{\gamma\b{\beta}}-R_{p\b{l}\gamma\b{\beta}}R_{k\b{q}\alpha\b{\sigma}}v^k\b{v}^lG^{
  \b{q}
  p
}
dz^\gamma\wedge d\b{z}^\sigma)\\
&=\b{\p}\p\log \det(G_{i\b{j}})+\b{\p}\p\log\det (\Omega_{\alpha\b{\beta}})\\
&=\b{\p}\p\log \left(\det(G_{i\b{j}})\cdot\det (\Omega_{\alpha\b{\beta}})\right).
\end{aligns}
 
Denote by $\iota: M\to E$ the natural embedding (as the zero section of $E$), then 
 \begin{align*}
 \iota^*(Ric^{\Omega})=\iota^*(\b{\p}\p\log \left(\det(G_{i\b{j}})\cdot\det (\Omega_{\alpha\b{\beta}})\right))=\b{\p}\p\log \left(\det(G_{i\b{j}})\cdot\det (g_{\alpha\b{\beta}})\right)	
 \end{align*}
is the $(1,1)$-form on $M$. 

In particular,  consider $E=TM$ and $(M, \omega)$ is  Teichm\"uller space with Weil-Petersson metric, that is, $(G_{i\bar{j}})=(g_{\alpha\b{\beta}})$ is  Weil-Petersson metric.  
 For any unit  vector $\xi\in E=TM$, i.e. $\|\xi\|^2=1$, then 
\begin{aligns}\label{1.7}
\iota^*(Ric^{\Omega})(\xi,\bar{\xi}) &=	\b{\p}\p\log \left(\det(g_{\alpha\bar{\beta}})\cdot\det (g_{\alpha\b{\beta}})\right)(\xi,\bar{\xi}) \\
&=2 Ric(\xi,\bar{\xi}),
\end{aligns}
where $Ric:=\b{\p}\p\log \det (g_{\alpha\b{\beta}})$ denotes the Ricci curvature of Weil-Petersson metric. From \cite[Lemma 4.6 (i)]{Wol}, the Ricci curvature of Weil-Petersson metric satisfies
\begin{align}\label{1.8}
	Ric(\xi,\bar{\xi})\leq -\frac{1}{2\pi(g-1)}.
\end{align}
where $g$ denotes the genus of Riemann surfaces. Substituting (\ref{1.8}) into (\ref{1.7}), we obtains
\begin{align*}
\iota^*(Ric^{\Omega})(\xi,\bar{\xi})\leq -\frac{1}{\pi(g-1)}.	
\end{align*}
Thus
\begin{prop}\label{prop3}
	Let $(M,\omega)$ be Teichm\"uller space  with the Weil-Petersson metric, and let $E=TM$ be the holomorphic tangent bundle. When restricting to $M$, the Ricci curvature of $\Omega$ is bounded from above by   $-\frac{1}{\pi(g-1)}$.
\end{prop}

\vspace{5mm}

Now we begin to prove our main theorem:
\begin{proof}[Proof of Theorem \ref{main theorem}]
	From Corollary \ref{bound},  the quasifuchsian space $QF(S)$ embeds into a neighborhood of a zero section 
in the holomorphic tangent bundle of $\mathcal{ T}(S)$ which is contained in a ball of radius $9\pi(g-1)$ in $L^2$-norm  on each fiber $B(X)$. 
Since the tangent bundle $\mc{B}(S)$ of $\mc{T}(S)$ with the Weil-Petersson metric $\omega_{WP}$ is Griffiths negative, so it defines a norm $G$ on 
$\mc{B}(S)$. Denote $\pi: \mc{B}(S)\to \mc{T}(S)$, then the following $(1,1)$-form
$$\Omega=\pi^*\omega_{WP}+\sqrt{-1}\p\b{\p}G$$
defines a mapping class group invariant K\"ahler metric on $\mc{B}(S)$ (see \cite{KZ}). From (\ref{metric1}), one sees that $\Omega$ is an extension of the Weil-Petersson metric $\omega_{WP}$.
From \cite[Theorem 1.5]{Mc}, the Weil-Petersson metric $\omega_{WP}$ has a bounded primitive with respect to Weil-Petersson metric. By Corollary \ref{cor1}, the  K\"ahler metric $\Omega$ also has a bounded primitive with respect to $\Omega$. And by Propositions \ref{prop2}, \ref{prop3}, the Chern curvature $R^\Omega$ of $\Omega$ is non-positive when evaluated on the tautological section $P$, and  its Ricci curvature is bounded from above by $-\frac{1}{\pi(g-1)}$ when restricted to Teichm\"uller space.
\end{proof}

\begin{rem} Finally to put our results in perspective
  we remark that
  the space $\cal P(S)$ of marked complex projective 
  structures is identified with the cotangent bundle of $\mathcal 
  T(S)$ and the  natural holonomy map $\cal P(S)\to 
\chi=  \chi(\pi_1(S), PSL(2,\C))$
  to the character variety is a local biholomorphic mapping
by the results of Earle-Hejhal-Hubbard  \cite{Ea, He, Hu} (see also \cite[Theorem 5.1]{Dumas}).
  Thus our constructions and 
   results  are also valid for $\cal P(S)$ and for  its image 
   in   $\chi$. The space $QF(S)$ of quasifuchsian representations
  is also an open subset of   $\chi$,
  $\mathcal T(S)\subset QF(S)
  \subset   \chi$, and
  it might be interesting to understand the geometry
  of  character variety   $ \chi$
  using our metric on these open subsets.
  \end{rem}

  The above remark applies also to
  the  Hitchin component for any real split simple Lie group
  $G$ of real rank
  two, namely $G=SL(3, R), Sp(2, \mathbb R), G_2$.
  Indeed Labourie \cite{La} generalized the construction
in \cite{KZ} of  K\"ahler
metric
for $SL(3, \mathbb R)$ 
to the above $G$.
In this case the Hitchin component
is proved to be a bundle over Teichm\"u{}ller space
with fiber being space of holomorphic differentials
of degree $3, 4, 6$, respectively. Hence we obtain 
\begin{cor} The curvature of  the K\"ahler metric on the Hitchin component for real split simple Lie groups of real rank 2 vanishes along vertical directions, and non-positive along tautological sections. 
\end{cor}

\end{document}